\documentclass[12pt]{amsart}
\usepackage{graphicx}
\usepackage{amsmath}
\usepackage{amsthm}
\usepackage{amsfonts}
\usepackage{amssymb}
\usepackage[usenames,dvipsnames]{color}
\usepackage{verbatim}
\usepackage{amscd}
\usepackage[latin1]{inputenc}
\usepackage{latexsym}
\usepackage{graphicx}

\newcommand{\zz}{\mathbb{Z}}

\newtheorem{theorem}{Theorem}[section]
\newtheorem{corollary}[theorem]{Corollary}
\newtheorem{lemma}[theorem]{Lemma}
\newtheorem{question}[theorem]{Question}

\theoremstyle{definition}
\newtheorem{definition}[theorem]{Definition}
\newtheorem{example}[theorem]{Example}
\newtheorem{remark}[theorem]{Remark}

\def\N{\mathbb{N}}

\def\Z{\mathbb{Z}}

\def\1{\mathbf{1}}

\DeclareMathOperator{\dist}{dist}

\DeclareMathOperator{\Var}{Var_{n,\alpha}}
\DeclareMathOperator{\vol}{vol}

\def\EE{\mathbb{E}_{n,\alpha}}
\def\PP{\mathbb{P}_{n,\alpha}}

\newenvironment{ProofOfFactorThm}[1]
{\par\vskip2\parsep\noindent{\sc Proof of Theorem\ \ref{Thm:Factors}. }}{{\hfill
$\Box$}
\par\vskip2\parsep}

\newenvironment{ProofOfEmbeddingThm}[1]
{\par\vskip2\parsep\noindent{\sc Proof of Theorem\ \ref{Thm:Embeddings}. }}{{\hfill
$\Box$}
\par\vskip2\parsep}

\newenvironment{ProofOfSecondMomentMethodThm}[1]
{\par\vskip2\parsep\noindent{\sc Proof of Theorem\ \ref{Thm:SecondMomentMethod}. }}{{\hfill
$\Box$}
\par\vskip2\parsep}

\title[Factors and embeddings for random $\mathbb{Z}^d$ SFTs]{Factor maps and embeddings for random $\mathbb{Z}^d$ shifts of finite type}

\begin{document}

\author{Kevin McGoff and Ronnie Pavlov}
\address{Kevin McGoff\\
Department of Mathematics\\
University of North Carolina at Charlotte\\
Charlotte, NC 28223}
\email{kmcgoff1@uncc.edu}
\urladdr{https://clas-math.uncc.edu/kevin-mcgoff/}
\address{Ronnie Pavlov\\
Department of Mathematics\\
University of Denver\\
2280 S. Vine St.\\
Denver, CO 80208}
\email{rpavlov@du.edu}
\urladdr{www.math.du.edu/$\sim$rpavlov/}
\thanks{The first author acknowledges the support of NSF grant DMS-1613261. The second author acknowledges the support of NSF grant DMS-1500685.}

\renewcommand{\subjclassname}{MSC 2010}

\begin{abstract}
For any $d \geq 1$, random $\mathbb{Z}^d$ shifts of finite type (SFTs) were defined in previous work of the authors. For a parameter $\alpha \in [0,1]$, an alphabet $\mathcal{A}$, and a scale $n \in \mathbb{N}$, one obtains a distribution of random $\mathbb{Z}^d$ SFTs by randomly and independently forbidding each pattern of shape $\{1,\dots,n\}^d$ with probability $1-\alpha$ from the full shift on $\mathcal{A}$. We prove two main results concerning random $\mathbb{Z}^d$ SFTs. First, we establish sufficient conditions on $\alpha$, $\mathcal{A}$, and a $\mathbb{Z}^d$ subshift $Y$ so that a random $\mathbb{Z}^d$ SFT factors onto $Y$ with probability tending to one as $n$ tends to infinity.  Second, we provide sufficient conditions on $\alpha$, $\mathcal{A}$ and a $\mathbb{Z}^d$ subshift $X$ so that $X$ embeds into a random $\mathbb{Z}^d$ SFT with probability tending to one as $n$ tends to infinity. 
\end{abstract}

\maketitle

\section{Introduction}


In this work we study mappings between certain $\mathbb{Z}^d$ topological dynamical systems called $\mathbb{Z}^d$ subshifts. 
In such dynamical systems, the points are elements of $\mathcal{A}^{\mathbb{Z}^d}$ for some finite set $\mathcal{A}$ (called an alphabet), and the dynamics are given by the collection of translations $\{\sigma_v\}$ for $v \in \mathbb{Z}^d$. 
A $\mathbb{Z}^d$ \textbf{subshift} is simply any subset of $\mathcal{A}^{\mathbb{Z}^d}$ that is closed (in the product topology) and invariant under all translations $\sigma_v$. 
A specific and well-studied class of subshifts are the so-called \textbf{shifts of finite type} or SFTs. 
A $\mathbb{Z}^d$ SFT is defined via a finite set of finite forbidden patterns $\mathcal{F}$; the SFT $X(\mathcal{F})$ induced by $\mathcal{F}$ is just the set of all $x \in \mathcal{A}^{\mathbb{Z}^d}$ that do not contain (translates of) any of the patterns from $\mathcal{F}$. (See Section~\ref{defs} for formal definitions.) In many ways, $\mathbb{Z}$ SFTs are fairly well-behaved objects. In contrast, $\mathbb{Z}^d$ SFTs exhibit a variety of pathological behaviors when $d > 1$, as evidenced by the significant literature (\cite{Hochman}, \cite{HM}, \cite{JV1}, \cite{JV2}) showing that for some properties, any behavior that is algorithmically computable can be realized within a $\mathbb{Z}^d$ SFT.

In \cite{McGoff}, the first author defined a probabilistic framework for $\mathbb{Z}$ SFTs, which was later extended to $\mathbb{Z}^d$ SFTs by both authors in \cite{McGoffPavlov}. As is often the case for ``measuring'' subsets of a countable set (e.g. various notions of density on the integers), the idea used was to consider limiting behavior over a class of exhaustive finite sets. Informally (see Section~\ref{defs} for full definitions), one begins with an alphabet $\mathcal{A}$ and a parameter $\alpha \in [0,1]$. Then, for each $n$, a $\mathbb{Z}^d$ SFT may be randomly selected by independently forbidding each pattern with shape $\{1,\dots,n\}^d$ with probability $1-\alpha$ and allowing it with probability $\alpha$. To understand how ``typical" $\mathbb{Z}^d$ SFTs behave in this sense, one then studies the limiting behavior of various events and random variables as $n \rightarrow \infty$. 

Within this framework, it was shown that even though some $\mathbb{Z}^d$ SFTs exhibit extremely complex or pathological behavior, a 
``typical'' $\mathbb{Z}^d$ SFT is much more tractable; see Subsection~\ref{randomsubsec} for details. 

Remark 5.2 in \cite{McGoffPavlov} asks whether this probabilistic approach could be used to extend some structural results about $\mathbb{Z}$ SFTs to ``typical'' $\mathbb{Z}^d$ SFTs, given that general extensions to all $\mathbb{Z}^d$ SFTs are known to be impossible. In this work, we focus specifically on extending some of the earliest work on $\mathbb{Z}$ SFTs, namely the existence of injective/surjective shift-commuting continuous maps (called embeddings/factors) between two of them. Our results (formally stated in Section \ref{Sect:MainResults}) provide further evidence that despite the existence of $\mathbb{Z}^d$ SFTs with pathological properties when $d>1$, ``typical" $\mathbb{Z}^d$ SFTs are generally well-behaved. Before stating these results, we present some additional context.

\subsection{Existing results on factors and embeddings}

Consider two $\mathbb{Z}^d$ subshifts $X$ and $Y$. In general, one would like to know when there exists an embedding from $X$ into $Y$ or a factor map from $X$ onto $Y$. 
For both embeddings and factor maps, there are two simple necessary conditions, given in terms of periodic points and topological entropy, as follows. 

For every $x \in \mathcal{A}^{\mathbb{Z}^d}$, we define its \textbf{period set} to be the set of $p \in \mathbb{Z}^d$ for which $\sigma_p x = x$, and then $x$ is said to be periodic if its period set is not $\{0\}$. By definition, if $\phi$ is a shift-commuting map on a subshift $X$, then $\phi(\sigma_p(x)) = \sigma_p(\phi(x))$ for all $x$ in $X$. Therefore, the image of a point with period set $S$ must have period set containing $S$, and if $\phi$ is injective, then the image must have period set exactly equal to $S$. These observations yield some immediate necessary conditions. In particular, for a factor map to exist from $X$ onto $Y$, it must be the case that for every $S \subset \mathbb{Z}^d$, if $X$ contains a point with period set $S$, then $Y$ contains a point with period set containing $S$. Also, for an embedding to exist from $X$ into $Y$, it must be the case that for every possible period set $S$, there are at least as many points in $Y$ with period set $S$ as there are in $X$. We note for future reference that whenever a point $x$ has finite orbit, its period set is necessarily a $d$-dimensional sublattice of $\mathbb{Z}^d$. In particular, when $d = 1$, any nonempty period set necessarily has the form $p\mathbb{Z}$ for some $p \in \N$; this $p$ is often called the \textbf{least period} of $x$ in the literature.

Topological entropy is an element of $\mathbb{R} \sqcup \{\infty\}$ associated to any topological dynamical system; see Section~\ref{defs} for a definition in the case of subshifts. We denote the topological entropy of a subshift $X$ by $h(X)$. It's well-known that if $X$ can be embedded into $Y$, then $h(X) \leq h(Y)$, and if $X$ factors onto $Y$, then $h(X) \geq h(Y)$, so in each case the associated inequality is a necessary condition for the existence of an embedding/factor map. 

A surprising fact is that when $d = 1$, these necessary conditions on periodic points and entropy are quite often nearly sufficient, as seen in the following classical results.

\begin{theorem}{\rm (\cite{krieger})}
Suppose $X$ is a $\mathbb{Z}$ subshift, $Y$ is a mixing $\mathbb{Z}$ SFT, $h(X) < h(Y)$, and for every possible period set $S$, there are at least as many points in $Y$ with period set $S$ as there are in $X$. Then $X$ embeds into $Y$.
\end{theorem}

\begin{theorem}{\rm (\cite{Boyle1983})}
Suppose $X$ and $Y$ are mixing $\mathbb{Z}$ SFTs, $h(X) > h(Y)$, and whenever $X$ contains a point with period set $S$, $Y$ contains a point with period set containing $S$. Then $X$ factors onto $Y$.
\end{theorem}

For the existence of factor maps, if $Y$ is a full shift, then it clearly contains points with every possible period set, and the entropy inequality and mixing assumption on $X$ can also be relaxed.

\begin{theorem}{\rm (\cite{Boyle1983}, \cite{Marcus1979})}
Suppose $X$ is a $\mathbb{Z}$ SFT and $Y$ is a full shift with $h(X) \geq h(Y)$. Then there exists a factor map from $X$ onto $Y$.
\end{theorem}

The picture is, however, quite different for $d > 1$. For instance, the following results were established in \cite{BPS}. 

\begin{theorem}{\rm (\cite{BPS})}\label{BPSthm1}
For every $d > 1$, there exist $\mathbb{Z}^d$ SFTs with arbitrarily large topological entropy that do not factor onto any nontrivial full shift. 
\end{theorem}

\begin{theorem}{\rm (\cite{BPS})}\label{BPSthm2}
For every $d > 1$, there exist $\mathbb{Z}^d$ SFTs $Y$ with arbitrarily large topological entropy and with a single zero entropy SFT $Z \subset Y$ that contains all minimal subsystems of $Y$.
\end{theorem}

Theorem \ref{BPSthm2} does not explicitly refer to embeddings, but the following corollary about embeddings is an immediate consequence.

\begin{corollary}
If $X$ contains a minimal subshift with positive entropy and $Y$ is a $\mathbb{Z}^d$ SFT as in Theorem~\ref{BPSthm2}, then there does not exist an embedding from $X$ to $Y$.
\end{corollary}

Despite the existence of systems with such seemingly pathological properties, there are some results that guarantee the existence of factor maps or embeddings between $\mathbb{Z}^d$ subshifts. However, such results generally require extremely strong hypotheses, including so-called uniform mixing conditions.
The most general results of this type in the literature are the following. (See Section~\ref{defs} for the definition of the finite extension property, \cite{BPS} for the definition of block gluing, and \cite{Lightwood2003} for the definition of square-filling mixing.) 

\begin{theorem}{\rm (\cite{McGoffPavlov_factors})}\label{BMPthm}
If $X$ is a $\mathbb{Z}^d$ block gluing subshift, $Y$ is a $\mathbb{Z}^d$ subshift with a fixed point and the finite extension property, and $h(X) > h(Y)$, then there exists a factor map from $X$ onto $Y$.
\end{theorem}

\begin{theorem}{\rm (\cite{Lightwood2003})}
If $X$ is a $\mathbb{Z}^2$ subshift with no periodic points (i.e. $\forall x \in X$ the period set of $x$ is $\{0\}$) and $Y$ is a square-filling mixing $\mathbb{Z}^2$ SFT with $h(X) < h(Y)$, then there exists an embedding of $X$ into $Y$.
\end{theorem}

\subsection{Known properties of random $\mathbb{Z}^d$ SFTs}\label{randomsubsec}

Here we informally summarize some relevant results from \cite{McGoffPavlov}; for formal statements and more information, see that work. In light of the main results of \cite{McGoffPavlov}, one may conclude that random $\mathbb{Z}^d$ SFTs exhibit qualitatively different behavior in two distinct parameter regimes. When $\alpha |\mathcal{A}| < 1$, a random $\mathbb{Z}^d$ SFT is empty with positive limiting probability, and when $\alpha |\mathcal{A}| \leq 1$, the entropy of a random $\mathbb{Z}^d$ SFT converges to $0$ in probability. On the other hand, when $\alpha |\mathcal{A}| > 1$, a ``typical'' random $\mathbb{Z}^d$ SFT is nonempty, has topological entropy close to $\log(\alpha |\mathcal{A}|) > 0$, and for large $m$, has many points with period set containing $\{me_i\}_{i=1}^d$ (with exponential growth rate equal to the topological entropy). Due to the necessary conditions on periodic points and entropy for factor maps and embeddings, the latter regime ($\alpha |\mathcal{A}| >1$) is more suitable for studying ``typical'' existence of factors/embeddings, and our main results will involve only this case.

\subsection{Factors and embeddings for random $\mathbb{Z}^d$ SFTs} \label{Sect:MainResults}

\textit{A priori} one might hope that factor maps and embeddings exist between a ``typical'' pair of random $\mathbb{Z}^d$ SFTs under appropriate hypotheses. Unfortunately, results of that type are impossible, due to the following fact, noted in \cite{McGoffPavlov}.

\begin{lemma}\label{perlemma}
Let $\mathcal{A}$ be an alphabet, and let $\alpha \in (0,1)$. Then for any $x$ in $\mathcal{A}^{\mathbb{Z}^d}$ with finite orbit $\mathcal{O}$, the limiting probability that $x$ is in the random $\mathbb{Z}^d$ SFT is 
$\alpha^{|\mathcal{O}|} \in (0,1)$.
\end{lemma} 

Therefore, if $X$ and $Y$ are independently chosen random $\mathbb{Z}^d$ SFTs, there is always a positive limiting probability that $X$ contains a fixed point (i.e. a point whose period set is $\mathbb{Z}^d$) and $Y$ does not, which trivially implies the following statement.

\begin{lemma}
For any alphabets $\mathcal{A}_X$ and $\mathcal{A}_Y$ and parameters $\alpha_X, \alpha_Y \in (0,1)$, if $X$ and $Y$ are independently chosen random $\mathbb{Z}^d$ SFTs, then there is a positive limiting probability that there does not exist a shift-commuting continuous map from $X$ to $Y$.
\end{lemma}

In light of this lemma, we cannot hope to prove results about ``typical'' existence of factor maps or embeddings between a pair of random $\mathbb{Z}^d$ SFTs. On the other hand, the two main results of this work guarantee ``typical'' existence of factor maps/embeddings when one of the subshifts involved is fixed and the other is a random $\mathbb{Z}^d$ SFT.

Our first main result is the following theorem about factor maps. We postpone the definition of the finite extension property to Section~\ref{defs}, but note that it is a conjugacy-invariant property of $\mathbb{Z}^d$ SFTs. In this theorem we use probabilistic notation, in which $X = X(\mathcal{F})$ is the random $\mathbb{Z}^d$ SFT obtained from the random set of forbidden words $\mathcal{F}$.

\begin{theorem}\label{Thm:Factors}
Let $d \geq 1$, and let $Y$ be a $\mathbb{Z}^d$ SFT that contains a fixed point and has the finite extension property. 
If $\log(\alpha_X |\mathcal{A}_X|) > h(Y)$, then there exists $\rho > 0$ such that for all large enough $n$,
\begin{equation*}
\mathbb{P}_{n,\alpha_X}\bigl(X \textrm{ factors onto } Y \bigr) \geq 1 - e^{-\rho n^d}. 
\end{equation*}
\end{theorem}

The assumption that $Y$ contains a fixed point is unavoidable; recall that $X$ has a fixed point with positive limiting probability by Lemma~\ref{perlemma}. As in Theorem~\ref{BMPthm}, we use the strong hypothesis of the finite extension property on the codomain $Y$. However, the only ``hypothesis" we need on $X$ is typicality.\\


Our second main result is the following theorem about embeddings. It requires a notion that we call the periodic marker condition, which is defined in Section \ref{defs}. In this theorem, the subshift $Y = Y(\mathcal{F})$ is the random $\mathbb{Z}^d$ SFT obtained from the random set of forbidden words $\mathcal{F}$.

\begin{theorem}\label{Thm:Embeddings}
Let $d \geq 1$, and let $X$ be a $\mathbb{Z}^d$ subshift satisfying the periodic marker condition with parameters $m_n = n$. If $\log(\alpha_Y |\mathcal{A}_Y|) > h(X)$, then there exists $\rho>0$ such that for all large enough $n$,
\begin{equation*}
 \mathbb{P}_{n,\alpha_Y} \bigl(X \textrm{ embeds into } Y \bigr) \geq 1 - e^{-\rho n^d}.
 \end{equation*}
\end{theorem}

We remark that any fixed subshift $X$ satisfying the conclusion of Theorem~\ref{Thm:Embeddings} cannot contain any finite orbits. Indeed, if $X$ had a point with finite orbit and period set $S$, then Lemma~\ref{perlemma} implies that there is a positive limiting probability that $Y$ contains no points with period set $S$, precluding an embedding into a ``typical" $\mathbb{Z}^d$ SFT. Although a formal definition of the periodic marker condition is deferred to Section~\ref{defs}, we note here that it is a sort of structured aperiodicity condition on $X$.

\

The proofs of both Theorems~\ref{Thm:Factors} and \ref{Thm:Embeddings} involve showing that large sets of patterns with prescribed occurrences of repeated subpatterns appear inside of random $\Z^d$ SFTs with high probability. To show the existence of such sets of patterns with high probability, we use a substantial generalization of the second moment arguments employed in \cite{McGoff,McGoffPavlov}. These arguments are carried out in Section \ref{Sect:SecondMomentMethod}. Then the proofs of Theorems \ref{Thm:Factors} and \ref{Thm:Embeddings} are given in Sections \ref{Sect:Factors} and \ref{Sect:Embeddings}, respectively.

\section{Preliminaries}\label{defs}

Here we establish the definitions and facts necessary to state and prove our main results.

\subsection{Subsets of $\mathbb{Z}^d$} \label{Sect:Subsets}

Let $\{e_i\}$ denote the standard basis in $\mathbb{Z}^d$. We use interval notation to denote subsets of $\mathbb{Z}$, e.g., $[3,7) = \{3,4,5,6\}$. For a subset $E$ of $\mathbb{R}^d$ and a vector $v$ in $\mathbb{R}^d$, let $E+v = \{u+v \ : \ u \in E\}$, and for two subsets $E$ and $F$ of $\mathbb{R}^d$, let $E + F = \{ u+v : u \in E, \, v \in F \}$. 

All references to distance refer to the $\ell_{\infty}$ metric: for $u, v \in \mathbb{R}^d$, we have $d(u,v) = \|u - v \|_{\infty} =  \max\{|u_i - v_i| \ : \ 1 \leq i \leq d\}$.
For $E \subset \mathbb{Z}^d$ and $r>0$, we define the outer $r$-boundary of $E$ to be
\begin{equation*}
\partial^{out}_r(E) = \{p \in \mathbb{Z}^d \setminus E : \dist(p,E) \leq r\},
\end{equation*}
and we define the inner $r$-boundary of $E$ to be
\begin{equation*}
\partial^{in}_r(E) = \{ p \in \mathbb{Z}^d \cap E : \dist(p, \mathbb{Z}^d \setminus E) \leq r\}.
\end{equation*}
We may also refer to the topological boundary of subsets $E$ of $\mathbb{R}^d$, meaning the closure of $E$ without the interior of $E$.

For every $n$, define $F_n = [0,n)^d \subset \mathbb{Z}^d$. For any finite set $D$ in $\mathbb{Z}^d$, let $\mathcal{C}_n(D)$ be the set of hypercubes with side length $n$ that are contained in $D$:
\begin{equation*}
\mathcal{C}_n(D) = \{ p + F_n : p \in \mathbb{Z}^d, \, p + F_n \subset D\}.
\end{equation*}

Here we record an elementary fact about overlapping hypercubes in the $\ell_{\infty}$ metric.
\begin{lemma} \label{Lemma:Woods}
Suppose that for $p,s,t \in \mathbb{Z}^d$, we have $(s+F_n) \cap (p+F_n) \neq \varnothing$ and $(t+F_n) \cap (p+F_n) \neq \varnothing$. Then $d(s,t) \leq 2n$.
\end{lemma} 
\begin{proof}
By considering each of the $d$ coordinates with respect to the standard basis of $\mathbb{Z}^d$, the statement reduces to the following fact: if $I_1, I_2$, and $J$ are three intervals of length $n$ in $\Z$ such that $I_1 \cap J \neq \varnothing$ and $I_2 \cap J \neq \varnothing$, then the left endpoints of $I_1$ and $I_2$ are within distance $2n$. This fact is a consequence of the triangle inequality for the absolute value.
\end{proof}

Later we will require some notation involving lattices in $\mathbb{Z}^d$. For our purposes, a lattice is a subgroup of $\mathbb{Z}^d$ with finite index. For any lattice $\Lambda$, one may select a basis $\{p_1,\dots,p_d\} \subset \Lambda$ (which must have cardinality $d$ since $\Lambda$ has finite index). Once a basis has been selected, one may define an associated fundamental domain:
\begin{equation*}
P = \Biggl\{ \sum_{i=1}^d s_i p_i : s_i \in [0,1) \Biggr\} \subset \mathbb{R}^d.
\end{equation*}
Furthermore, we let $\vol(\Lambda)$ denote the index of $\Lambda$ in $\mathbb{Z}^d$, and note that $\vol(\Lambda) = | P \cap \mathbb{Z}^d |$. 

\subsection{$\mathbb{Z}^d$ symbolic dynamics}

Here we provide some definitions from symbolic dynamics. Consider a natural number $d \geq 1$ and a finite set $\mathcal{A}$. 

\begin{definition}
A \textbf{pattern} over $\mathcal{A}$ is an element of $\mathcal{A}^S$ for some $S \subset \mathbb{Z}^d$, which is said to have \textbf{shape} $S$. If $S$ is finite, then any pattern with shape $S$ may be called a finite pattern.
\end{definition}

We consider patterns to be defined only up to translation, i.e., if $u \in \mathcal{A}^S$ for $S \subset \mathbb{Z}^d$ and $v \in \mathcal{A}^T$, where $T = S+p$ for some $p \in \mathbb{Z}^d$, then we write $u = v$ to mean that $u(s) = v(s+p)$ for each $s$ in $S$.


\begin{definition}
The \textbf{$\mathbb{Z}^d$-shift action} on $\mathcal{A}^{\mathbb{Z}^d}$, denoted by $\{\sigma_t\}_{t \in \mathbb{Z}^d}$, is defined by $(\sigma_t x)(s) = x(s+t)$ for $s,t \in \mathbb{Z}^d$. 
\end{definition}

We always think of $\mathcal{A}^{\mathbb{Z}^d}$ as being endowed with the product discrete topology, with respect to which it is compact. 

\begin{definition}
A \textbf{$\mathbb{Z}^d$ subshift} is a closed subset of $\mathcal{A}^{\mathbb{Z}^d}$ which is invariant under the $\mathbb{Z}^d$-shift action.
\end{definition}

\begin{definition} 
The \textbf{language} of a $\mathbb{Z}^d$ subshift $X$, denoted by $\mathcal{L}(X)$, is the set of all patterns with finite shape which appear in points of $X$. For any finite subset $S \subset \mathbb{Z}^d$, let $\mathcal{L}_S(X) := \mathcal{L}(X) \cap \mathcal{A}^S$, the set of finite patterns in the language of $X$ with shape $S$.
\end{definition}

Any subshift inherits a topology from $\mathcal{A}^{\mathbb{Z}^d}$, with respect to which it is compact. Each $\sigma_t$ is a homeomorphism on any $\mathbb{Z}^d$ subshift, and so any $\mathbb{Z}^d$ subshift, when paired with the $\mathbb{Z}^d$-shift action, is a topological dynamical system. 

Subshifts may also be defined in terms of disallowed patterns as follows. For any set $\mathcal{F}$ of finite patterns over $\mathcal{A}$, one can define the set 
$$X(\mathcal{F}) := \bigl\{x \in \mathcal{A}^{\mathbb{Z}^d} \ : \ x|_S \notin \mathcal{F} \ \text{ for all finite } S \subset \mathbb{Z}^d \bigr\}.$$
It is well known that any $X(\mathcal{F})$ is a $\mathbb{Z}^d$ subshift, and any $\mathbb{Z}^d$ subshift may be presented in this way. 

\begin{definition}
A \textbf{$\mathbb{Z}^d$ shift of finite type (SFT)} is a $\mathbb{Z}^d$ subshift equal to $X(\mathcal{F})$ for some finite set $\mathcal{F}$ of forbidden finite patterns. 
\end{definition}

For the purposes of this paper, we only consider factor maps and embeddings between subshifts.

\begin{definition}
A (topological) \textbf{factor map} is any surjective continuous map $\phi$ from a $\mathbb{Z}^d$ subshift $X$ to a $\mathbb{Z}^d$ subshift $Y$ that commutes with the $\mathbb{Z}^d$ shift action. 
\end{definition}

\begin{definition}
A (topological) \textbf{embedding} is any injective continuous map $\phi$ from a $\mathbb{Z}^d$ subshift $X$ into a $\mathbb{Z}^d$ subshift $Y$ that commutes with the $\mathbb{Z}^d$ shift action. 
\end{definition}

In the case of subshifts, topological entropy may be defined as follows.
\begin{definition}\label{entdef}
The \textbf{topological entropy} of a $\zz^d$ subshift $X$ is
\[
h(X) := \lim_{n_1, \ldots, n_d \rightarrow \infty} \frac{1}{\prod_{i=1}^d n_i} \log \biggl|\mathcal{L} _{\prod_{i=1}^d [1,n_i]}(X) \biggr|.
\]
\end{definition}

\

Next we define the finite extension property from \cite{McGoffPavlov_factors}, which appears in our main result about factors of random SFTs (Theorem \ref{Thm:Factors}).

\begin{definition}\label{extprop}
For $g \in \mathbb{N}$, a $\mathbb{Z}^d$ SFT $X$ has the \textbf{$g$-extension property} if there exists a finite set $\mathcal{F}$ of forbidden finite patterns inducing $X$ with the following property: if a pattern $w$ with shape $S$ can be extended to a pattern on $S + [-g,g]^d$ that does not contain any patterns from $\mathcal{F}$, then $w \in \mathcal{L}(X)$, i.e., it can be extended to a point on all of $\mathbb{Z}^d$ that does not contain any patterns from $\mathcal{F}$. We say that $X$ has the \textbf{finite extension property} if it has the $g$-extension property for some $g$ in $\mathbb{N}$.
\end{definition}

\subsection{Allowed patterns and random $\mathbb{Z}^d$ SFTs}

Recall that for any finite set $D$ in $\mathbb{Z}^d$, we use $\mathcal{C}_n(D)$ to denote the set of hypercubes with side length $n$ that are contained in $D$. For any pattern $u \in \mathcal{A}^D$, let $W_n(u)$ denote the set of patterns with shape $F_n$ that appear in $u$:
\begin{equation*}
W_n(u) = \bigl\{ u|_S : S \in \mathcal{C}_n(D) \bigr\}.
\end{equation*}
As we only consider patterns to be defined up to translation, we always have $W_n(u) \subset \mathcal{A}^{F_n}$.

If $G$ is any set of patterns (of any shape) and $\mathcal{F} \subset \mathcal{A}^{F_n}$, then let $G(\mathcal{F})$ be the set of patterns in $G$ that do not contain any pattern from $\mathcal{F}$ as a subpattern:
\begin{equation*}
G(\mathcal{F}) = \bigl\{ u \in G : W_n(u) \cap \mathcal{F} = \varnothing \bigr\}.
\end{equation*}
When $\mathcal{F}$ is a set of forbidden patterns, elements of $G(\mathcal{F})$ may be called allowed patterns.

\

Finally, we recall the probabilistic framework used to study random $\mathbb{Z}^d$ SFTs in \cite{McGoffPavlov}.
For each $\alpha$ in the unit interval $[0,1]$ and $n \in \N$, we define a probability measure $\mathbb{P}_{n,\alpha}$ on the power set of $\mathcal{A}^{F_n}$ as follows. For any set $\mathcal{F} \subset \mathcal{A}^{F_n}$, let
\begin{equation*}
\mathbb{P}_{n,\alpha}( \{\mathcal{F}\}) = \alpha^{|\mathcal{A}|^{n^d} - |\mathcal{F}|} (1-\alpha)^{|\mathcal{F}|}.
\end{equation*}
Then $\mathbb{P}_{n,\alpha}$ induces a probability measure on $\Z^d$ SFTs by associating the set $\mathcal{F}$ to the SFT $X = X(\mathcal{F})$ defined by forbidding all the patterns in $\mathcal{F}$, i.e.
\begin{equation*}
X = \bigl\{ x \in \mathcal{A}^{\Z^d} : \forall p \in \Z^d, \, \sigma_p(x)|_{F_n} \notin \mathcal{F} \bigr\}.
\end{equation*}
When $\mathcal{F}$ is drawn according to $\mathbb{P}_{n,\alpha}$, we refer to the corresponding subshift $X$ as a random $\mathbb{Z}^d$ SFT.
Whenever we refer to an event $\mathcal{E}$ having \textbf{limiting probability one}, we mean that $\alpha$ and $\mathcal{A}$ are fixed and $\mathbb{P}_{n,\alpha}(\mathcal{E})$ tends to one as $n$ tends to infinity.

\subsection{Regular sequences of lattices}

Here we establish the properties of a sequence of lattices that we require for our second moment argument. 
Recall that both the fundamental domain $P_n$ associated to a lattice $\Lambda_n$ in $\mathbb{Z}^d$ and the volume of $\Lambda_n$ were defined in Section \ref{Sect:Subsets}.
\begin{definition} \label{Defn:PeriodicMarkerCondition}
Let $\{m_n\}_n$ be a sequence of natural numbers such that $m_n \geq n$ for all $n$.
A sequence $\{\Lambda_n\}_n$ of lattices in $\mathbb{Z}^d$ is said to be \textbf{$\{m_n\}$-regular} if there exists a sequence of associated fundamental domains $\{P_n\}_n$ satisfying 
\begin{enumerate}
\item[(P1)] (subexponential growth) 
\begin{equation*}
\lim_n \frac{1}{n} \log \vol(\Lambda_n) = 0,
\end{equation*}
and
\item[(P2)] (small outer boundaries) 
\begin{equation*}
\lim_n \frac{ \bigl| \partial_{m_n}^{out}(P_n \cap \mathbb{Z}^d) 
\bigr|} {\vol(\Lambda_n)} = 0.
\end{equation*}
\end{enumerate}
\end{definition}

\begin{example}\label{cubeexample}
Suppose $\{k_n\}_n$ and $\{m_n\}_n$ are sequences of natural numbers satisfying $m_n \geq n$ for all $n$,  $\frac{m_n}{k_n} \to 0$, and $n^{-1} \log k_n \to 0$. If $\Lambda_n = k_n \mathbb{Z}^d$ with fundamental domain $P_n = [0,k_n)^d$ (here the interval notation refers to the subset of $\mathbb{R}$, not $\mathbb{Z}$), then $\{\Lambda_n\}_n$ is an $\{m_n\}$-regular sequence of lattices.
\end{example}

Regular sequences of lattices also satisfy the following properties.
\begin{lemma} \label{Eqn:FDT}
If $\{\Lambda_n\}$ is an $\{m_n\}$-regular sequence of lattices, then $\{\Lambda_n\}_n$ also satisfies
\begin{enumerate}
\item[(P3)] (small inner boundaries)
\begin{equation*}
\lim_n \frac{ \bigl| \partial_{m_n}^{in}(P_n \cap \mathbb{Z}^d) 
\bigr|}{\vol(\Lambda_n)} = 0,
\end{equation*}
and
\item[(P4)] (separation of lattice points)
for all large enough $n$, if $p, q \in \Lambda_n$ and $p \neq q$, then $d(p,q) >2n$.
\end{enumerate}
\end{lemma}
\begin{proof}
Let $\{\Lambda_n\}_n$ be as in the hypotheses, and for each $n$, let $\eta_n : \mathbb{Z}^d \to P_n$ be the map that sends $p$ to the unique element $q \in P_n$ such that $p - q \in \Lambda_n$. For the sake of notation, let $E_n = P_n \cap \mathbb{Z}^d$.

First, we claim that $\partial^{in}_{m_n}(E_n) \subset \eta_n( \partial^{out}_{m_n}(E_n) )$. Indeed, let $v \in \partial^{in}_{m_n}(E_n)$, i.e. $v \in E_n$ and $\dist(v,P_n^c) \leq m_n$. Then there exists a vector $q$ such that $d(v,q) \leq m_n$ and $q$ lies on the topological boundary of $P_n$. Hence $q$ may be written as $\sum_i s_i p_i$, where $s_i \in [0,1]$ and there exists at least one index $i_0$ such that $s_{i_0} \in \{0,1\}$. Let $v' = v+(-1)^{s_{i_0}}p_{i_0}$ and $q' = q+(-1)^{s_{i_0}}p_{i_0}$. Since $v \in P_n$ and $v' - v \in \Lambda_n$, we have $\eta_n(v') = v$. Also, $v' \in \mathbb{Z}^d \setminus P_n$, and $q'$ is in the topological boundary of $P_n$. Finally, observe that $\dist(v',P_n) \leq d(v',q') = d(v,q) \leq m_n$. Hence $v' \in \partial^{out}_{m_n}(E_n)$. Since $v$ was arbitrary, we have verified the claim. 

By the claim, $|\partial^{in}_{m_n}(E_n)| \leq |\partial^{out}_{m_n}(E_n)|$. Property (P3) is a direct consequence of (P2) and this inequality.

Now we prove (P4). Suppose $p,q \in \Lambda_n$ and $p \neq q$; we can clearly assume without loss of generality that $p \in \overline{P_n}$ by translation, and then it suffices to treat only the case where $q \in \overline{P_n}$, since all other $q \in \Lambda_n$ are at least as far away. Now suppose for a contradiction that $d(p,q) \leq 2n$. Let $F$ be a face of $\overline{P_n}$ containing $p$ but not $q$, and let $v$ be a unit vector perpendicular to $F$ such that $v \cdot q >0$. For $j \in \{0,1\}$, let $Q_j = \{ s \in P_n \cap \mathbb{Z}^d : s \cdot v \in [jn,(j+1)n] \}$. Since $P_n$ is a parallelotope and $d(p,q) \leq 2n$, we must have $P_n \cap \mathbb{Z}^d = Q_0 \cup Q_1$. Also, note that $Q_j \subset \partial^{in}_{n}(E_n)$. Hence $P_n \cap \mathbb{Z}^d \subset \partial^{in}_{n}(E_n)$. However (P3) implies that for all large enough $n$, this containment does not hold. Thus, for all large enough $n$, we must have $d(p,q) > 2n$.
\end{proof}

\begin{definition} \label{Defn:Roy}
Consider an $\{m_n\}$-regular sequence of lattices $\{\Lambda_n\}_n$, and let $\{P_n\}_n$ be an associated sequence of fundamental domains satisfying (P1) and (P2).
Define $D_n$ to be the set of $p \in \Z^d$ such that $\dist(p,P_n) \leq m_n$, and define $\eta_n : D_n \to P_n$ to be the map that sends $p$ to the unique element $q$ of $P_n$ such that $p - q \in \Lambda_n$.
Also, to simplify notation, let $E_n = \eta_n(D_n) = P_n \cap \Z^d$.
\end{definition}

\begin{lemma} 
Let $\{\Lambda_n\}_n$ be an $\{m_n\}$-regular sequence of lattices with associated objects as in Definition \ref{Defn:Roy}. Then the following properties are also satisfied.
\begin{itemize}
\item[(P5)] The sequence of maps $\{\eta_n\}_n$ is eventually uniformly bounded-to-one: there exists $K \geq 1$ such that for all large enough $n$ and for all $r \in E_n$, we have $|\eta_n^{-1}(r)| \leq K$.
\item[(P6)] For all large enough $n$, the map $\eta_n$ is locally injective on hypercubes: for any $S \in \mathcal{C}_n(D_n)$, we have $|\eta_n(S)| = n^d$.
\end{itemize}
\end{lemma}
\begin{proof}
Let $K = 3^d$. We will show that (P5) holds with this $K$. First we prove a preliminary statement. Let
\begin{equation*}
B_n = \biggl\{ \sum_i v_i p_i + r : v_i \in \{-1,0,1\}, r \in P_n \biggr\}.
\end{equation*}
We claim that if $(v+P_n) \bigcap P_n \neq \varnothing$, then $v + P_n \subseteq B_n$. To see this, suppose that $(v+P_n) \bigcap P_n \neq \varnothing$. Then there exists a (corner) vector $c = \sum_i c_i p_i$ with $c_i \in \{0,1\}$ such that $v+c \in P_n$. Hence $v+c = r$ for some $r \in P_n$. Then for any $v+s$, with $s \in P_n$, we have $v+s = (r-c)+s = \sum_i (r_i - c_i + s_i)p_i$. Note that $r_i - c_i+s_i \in [-1,2)$, and hence $v +s \in B$.

Now we claim that for all large enough $n$, we have $D_n \subset B_n$. Let $q \in D_n$, which by definition means that there is a vector $p \in P_n$ such that $d(q,p) \leq m_n$. Let $v = q-p$. For all $u \in v + P_n$, we have $u = v+r$ for some $r \in P_n$, and then $d(u,r) = \| v+r - r\|_{\infty} = d(q,p) \leq m_n$. Hence 
\begin{equation*}
v+P_n \subset \biggl\{ u : \dist(u,P_n) \leq m_n \biggr\}.
\end{equation*}
If $(v+P_n) \bigcap P_n = \varnothing$ for infinitely many $n$, then this contradicts the small outer boundaries property (P2). Thus, for all large enough $n$, we have $(v+P_n) \bigcap P_n \neq \varnothing$. Then by our previous claim, we have that $v+P_n \subset B_n$, which in particular gives that $q = v+p \in B_n$. As $q \in D_n$ was arbitrary, we obtain that $D_n \subset B_n$.

For each element $p \in B_n$, by definition, we have that $p-\eta_n(p)$ has the form $\sum_i v_i p_i$, where $v_i \in \{-1,0,1\}$. Therefore we conclude that $|\eta_n^{-1}(r)| \leq 3^d$.

Now we show (P6). 
Let $S \in \mathcal{C}_n(D_n)$. In particular, let $S = v + F_n$. 
For $p \in \Lambda_n$, let
\begin{equation*}
S(p) = \bigl\{ t \in S : t - \eta_n(t) = -p \bigr\}.
\end{equation*}
Note that $\{S(p)\}_{p \in \Lambda_n}$ is a partition of $S$. By definition of $S(p)$, we have $\eta_n(S(p)) = p + S(p)$, which shows that $\eta_n$ is injective on $S(p)$. Furthermore,  observe that $\eta_n(S(p)) = p+S(p) \subset p + v +F_n$.  By separation of lattice points (P4), for all large enough $n$, if $p,q \in \Lambda_n$ and $p \neq q$, then $(p+F_n) \bigcap (q + F_n) = \varnothing$, which implies that $(p + v + F_n) \bigcap (q + v + F_n) = \varnothing$, and hence $\eta_n(S(p)) \bigcap \eta_n(S(q)) = \varnothing$. Therefore $\eta_n$ is injective on $S$.
\end{proof}

\subsection{The periodic marker condition}

In this section, we define the periodic marker condition, which appears in our main result on embeddings (Theorem \ref{Thm:Embeddings}).
Let $\mathcal{O}$ be a finite orbit in $\mathcal{A}^{\mathbb{Z}^d}$. Then there is a unique lattice $\Lambda$ in $\mathbb{Z}^d$ such that for any $x \in \mathcal{O}$, the period set of $x$ is $\Lambda$ (\textit{i.e.} $\{ p \in \mathbb{Z}^d : \sigma_p(x) = x \} = \Lambda$). 

\begin{definition}
A $\mathbb{Z}^d$ subshift $X$ satisfies the \textbf{periodic marker condition} with parameters $\{m_n\}_n$ if $X$ factors onto a sequence of finite orbits $\{\mathcal{O}_n\}_n$ such that the associated sequence of lattices $\{\Lambda_n\}_n$ is an $\{m_n\}$-regular sequence of lattices.
\end{definition}

\begin{example}
Suppose $X$ is a substitutionally defined $\mathbb{Z}^d$ subshift (see \cite{frank}) induced by a substitution $\tau$ with the following properties:

\begin{itemize}
\item[(1)] there is a rectangular prism $R$ so that for each $a \in \mathcal{A}$, the pattern $\tau(a)$ has shape $R$;


\item[(2)] $\tau$ is uniquely decomposable as defined in \cite{solomyak}; this means that there exists $N$ so that knowledge of $x([-N,N]^d)$ determines whether $x(0)$ occupies the least coordinate (in the lexicographic order) of a block $\tau(a)$. 
\end{itemize}
Then $X$ satisfies the periodic marker condition with parameters $m_n = n$. 
\end{example}
\begin{proof}
Suppose that $\tau$ has the stated properties. For every $n$, define $R_n$ to be the shape of $\tau^n(a)$ for any $a \in A$ (which is independent of the choice of $a$). The dimensions of $R_n$ are just the $n$th powers of the dimensions of $R$. Since $\tau$ is uniquely decomposable, for every $n$ there exists $N(n)$ so that $x([-N(n),N(n)]^d)$ determines whether $x(0)$ is the least coordinate 
(lexicographically) of some $\tau^n(a)$. This means that for every $n$, there is a factor map $\phi_n$ that assigns $1$ to locations in any $x$ that are least coordinates of $\tau^n(a)$ and $0$ elsewhere. Then clearly $\phi_n(X)$ is just a periodic orbit with fundamental domain $R_n$. 

Choose any monotone increasing sequence $k_n$ with $\frac{k_n}{n} \rightarrow 0$ and $\frac{k_n}{\ln n} \rightarrow \infty$. Then the reader may check that the sequence $\{\phi_{k_n}(X)\}_n$ yields an $\{m_n\}$-regular sequence of lattices, and so $X$ satisfies the periodic marker condition with parameters $m_n = n$.
\end{proof}

\begin{remark} 
Any $\mathbb{Z}^d$ subshift that factors onto such a substitutionally defined $\mathbb{Z}^d$ subshift with properties (1) and (2) above (including the $\mathbb{Z}^d$ SFTs defined in \cite{mozes} or the examples of $\mathbb{Z}^d$ SFTs with arbitrary right recursively enumerable entropy defined in \cite{HM}) obviously also satisfies the periodic marker condition for $m_n = n$.
\end{remark}

\section{Second moment method} \label{Sect:SecondMomentMethod}

In this section we generalize the second moment arguments presented in previous work on random $\mathbb{Z}^d$ SFTs. The goal is to show that with high probability, the random $\mathbb{Z}^d$ SFT $X$ must have a large collection of patterns with prescribed structure of repetition among their subpatterns. Such collections of patterns will form the basis of our constructions of factor maps and embeddings in later sections. To establish the existence of these collections of patterns, we consider the random variable that counts how many patterns with prescribed repetition structure are allowed. Then we pursue estimates on the mean and variance of this random variable.

\subsection{Generalities}

Fix a natural number $d>1$, a finite set $\mathcal{A}$, and $\alpha \in (|\mathcal{A}|^{-1}, 1]$. We consider random $\mathbb{Z}^d$ SFTs in $\mathcal{A}^{\mathbb{Z}^d}$ distributed according to $\mathbb{P}_{n,\alpha}$. In \cite{McGoffPavlov}, a second moment method was used to show that with high probability, there exist many patterns on a large cube that can be freely stitched together to form allowable tilings of $\mathbb{Z}^d$. 
Specifically, we showed that with high probability there must be a very large collection of patterns in $\mathcal{L}(X)$ with shape $F_k$ that all agree on the $n$-boundary of $F_k$, and the common pattern on the $n$-boundary of $F_k$ may be taken to have the same subpattern on all pairs of parallel faces. (Such patterns were called periodic boundaries in \cite{McGoffPavlov}.) 
In this work, we strengthen this second moment argument to include more general domains than $F_k$ and to accommodate more stringent demands on the $F_n$-subpatterns, not only asserting the existence of repeated words with shape $F_n$ at more (pairs of) locations, but also excluding repeats at all other (pairs of) locations. 




Let us now begin to describe the prescribed repeat structure we wish to find inside many words of a random $\mathbb{Z}^d$ SFT with high probability. 
Consider an $\{m_n\}$-regular sequence of lattices $\{\Lambda_n\}_n$ satisfying $m_n \geq n$ for all $n$, along with the associated objects in Definition \ref{Defn:Roy}.
Let $G_n^0$ be the following set of patterns on $D_n$:
\begin{equation*}
G_n^0 = \bigl\{ u \in \mathcal{A}^{D_n} : \text{ if } \eta_n(p)  = \eta_n(q), \text{ then } u_p = u_q \bigr\}.
\end{equation*}
The reader may check that $|G_n^0| = |\mathcal{A}|^{|E_n|}$, since the map that sends $u \in G_n^0$ to $u|_{E_n}$ is a bijection onto $\mathcal{A}^{E_n}$.
Also observe that if $p \in D_n$, then $\eta_n(p) + F_n \subset E_n \cup  \partial^{out}_{n}(E_n) \subset D_n$ (since $m_n \geq n$). Therefore if $u \in G_n^0$ and $w = u|_{p+F_n}$, then $\eta_n(p) \in E_n$ and $w = u|_{\eta_n(p) +F_n}$. Thus, for each $u \in G_n^0$, we have that $|W_n(u)| \leq |E_n|$.

%

Let $G_n$ be the set of patterns in $G_n^0$ that achieve this upper bound:
\begin{equation*}
G_n = \bigl\{ u \in G_n^0 : |W_n(u)| = |E_n| \bigr\}.
\end{equation*}

The main result in this section is the following theorem, which states that for $\alpha > |\mathcal{A}|^{-1}$, with high probability, many patterns from $G_n$ avoid the randomly chosen set of forbidden words. Recall that for $\mathcal{F} \subset \mathcal{A}^{F_n}$, the set $G_n(\mathcal{F})$ consists of the patterns in $G_n$ that avoid all patterns in $\mathcal{F}$.

\begin{theorem} \label{Thm:SecondMomentMethod}
Let $d \geq 1$, $\alpha \in ( |\mathcal{A}|^{-1}, 1]$, and let $\{\Lambda_n\}_n$ be an $\{m_n\}$-regular sequence of lattices with associated objects as in Definition \ref{Defn:Roy}.
Then for any $c \in (0,1)$, there exists $\rho >0$ such that for all large enough $n$,
\begin{equation*}
\mathbb{P}_{n,\alpha} \Bigl( | G_n(\mathcal{F}) | \geq c (\alpha |\mathcal{A}|)^{|E_n|} \Bigr) \geq 1- e^{-\rho n^d}.
\end{equation*}
\end{theorem}

This theorem plays a central role in the proofs of our main results on factors and embeddings (Theorems \ref{Thm:Factors} and \ref{Thm:Embeddings}).
Before proving Theorem \ref{Thm:SecondMomentMethod}, we first establish some structural lemmas. The following lemma is very similar to Lemma 3.6 in \cite{McGoffPavlov}; nonetheless, we provide a proof here for completeness.

\begin{lemma} \label{Lemma:LukeMaye}
Under the hypotheses of Theorem \ref{Thm:SecondMomentMethod}, for all large enough $n$, the following property holds: if $p, q \in E_n$, $p \neq q$, $T = \eta_n(q+F_n)$, and $w \in \mathcal{A}^{E_n \setminus T}$, then there is at most one $u \in G_n^0$ such that $u|_{p + F_n} = u|_{q +F_n}$ and $u|_{E_n \setminus T} = w$.
\end{lemma}
\begin{proof}
By separation of lattice points (P4), for all large enough $n$, the distance between any two lattice points is strictly greater than $2n$.
Let $n$ be large enough for this inequality to hold.
Now let $p,q,T$, and $w$ be as in the statement of the lemma. Note for future reference that since $p,q \in E_n$ we have that $p-q \notin \Lambda_n$. Suppose that $u^1,u^2 \in G_n^0$ both satisfy the properties in the conclusion of the lemma. We will show that $u^1 = u^2$.

First, notice that for all $p' \in D_n$ such that $\eta_n(p') \notin T$, we have 
\begin{equation} \label{Eqn:Soph}
u^1(p') = u^1(\eta_n(p')) = w(\eta_n(p')) = u^2(\eta_n(p')) = u^2(p'),
\end{equation}
where the first and last equalities hold since $u^1,u^2 \in G_n^0$.

We claim that at most one $s \in \Lambda_n$ satisfies $(s+q+F_n) \cap (p+F_n) \neq \varnothing$. Indeed, suppose that for $s,t \in \Lambda_n$, we have $(s+q+F_n) \cap (p+F_n) \neq \varnothing$ and $(t+q+F_n) \cap (p+F_n) \neq \varnothing$. Then by Lemma \ref{Lemma:Woods}, we obtain that $d(s+q,t+q) \leq 2n$, and therefore $d(s,t) \leq 2n$. By our choice of $n$, we conclude that $s = t$. 

Furthermore, if $s \in \Lambda_n$ satisfies $(s+q+F_n) \cap (p+F_n) \neq \varnothing$, then $s+q \neq p$, since $p-q \notin \Lambda_n$. Therefore there exists a corner $C$ of $p+F_n$ contained in $D_n \setminus (\Lambda_n + (q+F_n))$.

Place a lexicographic ordering on $p + F_n$ in which $C$ is the minimal element. Now we claim by induction on the lexicographic ordering that $u^1(p') = u^2(p')$ for all $p' \in p+F_n$. For the base case, Equation (\ref{Eqn:Soph}) gives that $u^1(C) = u^2(C)$, since our choice of $C$ implies that $\eta_n(C) \notin T$. Now suppose that for all $p'' < p'$ in $p+F_n$, we have $u^1(p'') = u^2(p'')$. If $\eta_n(p') \notin T$, then Equation (\ref{Eqn:Soph}) gives that $u^1(p') = u^2(p')$. On the other hand, consider $p' \in p+F_n$ such that $\eta_n(p') \in T$.  Then $p' \in s + q + F_n$ where $s$ is the unique element of $\Lambda_n$ such that $(s+q+F_n) \cap (p+F_n) \neq \varnothing$. Therefore $p' = s+q+t$ for some $t \in F_n$. Let $p'' = p +t$. Then by our choice of ordering on $p+F_n$, we have $p'' < p'$. Furthermore, by the induction hypothesis and the fact that $u^i|_{p+F_n} = u^i|_{q+F_n} = u^2|_{s+q+F_n}$, we must have $u^1(p') = u^1(p'') = u^2(p'') = u^2(p')$. This completes our induction and establishes that $u^1(p) = u^2(p)$.

Then by the fact that $u^i|_{p+F_n} = u^i|_{q+F_n}$, we see $u^1|_{q+F_n} = u^1|_{p+F_n} = u^2|_{p+F_n} = u^2|_{q+F_n}$. Finally, since $u^i \in G_n^0$, we conclude that $u^1|_T = u^2|_T$, which, together with Equation (\ref{Eqn:Soph}), gives that $u^1 = u^2$.
\end{proof}

Next we establish a simple lower bound on the cardinality of $G_n$.
\begin{lemma} \label{Lemma:Halloween}
Under the hypotheses of Theorem \ref{Thm:SecondMomentMethod},
\begin{equation*}
|G_n| \geq  |\mathcal{A}|^{|E_n|} \bigl( 1 - |E_n|^2 \cdot |\mathcal{A}|^{-n^d} \bigr).
\end{equation*}
\end{lemma}
\begin{proof}
Let $v \in G_n^0 \setminus G_n$. As $v$ is not in $G_n$, the map that sends $p \in E_n$ to $v|_{p+F_n}$ is not injective. Hence, there are elements $p,q \in E_n$ such that $p \neq q$ and $v|_{p+F_n} = v|_{q+F_n}$, and we select the lexicographically minimal such pair for definiteness. The set $T = \eta_n(p+F_n) \subset E_n$ has cardinality $n^d$ by (P6). Furthermore, the map $v \mapsto (p,q,v|_{E_n \setminus T})$ is injective by Lemma \ref{Lemma:LukeMaye}, and its image clearly has cardinality bounded above by $|E_n|^2 \cdot |\mathcal{A}|^{|E_n|-n^d}$. Hence, we have established that
\begin{equation*} 
|G_n^0 \setminus G_n| \leq |E_n|^2 \cdot |\mathcal{A}|^{|E_n|-n^d},
\end{equation*}
The conclusion of the lemma is a direct consequence of this inequality and the fact that $|G_n^0| = |\mathcal{A}|^{|E_n|}$.
\end{proof}

Here we set some notation for the remaining results in this section.
For each $n$, consider the random variable
\begin{equation*}
\varphi_n = \sum_{u \in G_n} \xi_u,
\end{equation*} 
where $\xi_u$ is the indicator function of the event that $u$ is allowed (i.e., $\xi_u$ is $1$ if $u$ contains no patterns in the random forbidden set $\mathcal{F}$ and $0$ otherwise). If $\mathcal{F}$ is a forbidden set of patterns drawn at random from $\mathbb{P}_{n,\alpha}$, then $\varphi_n$ is the number of patterns from $G_n$ that do not contain a pattern from the forbidden set $\mathcal{F}$, i.e. $\varphi_n = |G_n(\mathcal{F})|$. The expectation of $\varphi_n$ satisfies
\begin{equation} \label{Eqn:Expectation}
\EE \bigl[ \varphi_n \bigr] = \sum_{u \in G_n} \EE \bigl[ \xi_u \big] = \sum_{u \in G_n} \alpha^{|W_n(u)|} = \alpha^{|E_n|} \cdot |G_n|.
\end{equation}
Similarly, for the variance of $\varphi_n$, we have
\begin{align*}
\Var \bigl[ \varphi_n \bigr]  = \alpha^{2|E_n|} \sum_{\substack{u,v \in G_n \\ W_n(u) \cap W_n(v) \neq \varnothing }} \bigl(\alpha^{-| W_n(u) \cap W_n(v) |} - 1 \bigr).
\end{align*}
It is convenient to rewrite the variance as follows.
Let 
\begin{equation*}
V_n  = \bigl\{ (u,v) \in G_n \times G_n : W_n(u) \cap W_n(v) \neq \varnothing \bigr\},
\end{equation*}
and
\begin{equation*}
V_{n,r} = \bigl\{ (u,v) \in V_n : |W_n(u) \cap W_n(v)| = r \bigr\}.
\end{equation*}
Then
\begin{align} \label{Eqn:Variance}
\Var \bigl[ \varphi \bigr]  = \alpha^{2|E_n|}  \sum_{r=1}^{|E_n|} |V_{n,r}| \, (\alpha^{-r}-1).
\end{align}

We partition $V_{n,r}$ in the following way. Given $(u,v) \in V_{n,r}$, define a \textbf{cross-repeat} between $u$ and $v$ to be a pair 
$(p,q) \in E_n \times E_n$ for which $u|_{p+F_n} = v|_{q+F_n}$. Let $R_0(u,v)$ be the union of the sets $q+F_n$ over all cross repeats $(p,q)$, and let $R(u,v) = \eta_n(R_0(u,v))$. For every $a$, define
\begin{equation*}
V_{n,r,a} = \bigl\{ (u,v) \in V_{n,r} : |R(u,v)| = a \bigr\}.
\end{equation*}

Let $(u,v) \in V_{n,r,a}$, and let $J = J(u,v)$ be the set of cross-repeats between $u$ and $v$.
Consider the map $\psi : J \to R(u,v)$ given by $\psi(p,q) = q$. By the definition of a cross-repeat, $\psi$ is well-defined. Furthermore, observe that $\psi$ is injective, since $u$ and $v$ are each in $G_n$. Since $(u,v) \in V_{n,r,a}$, we also have $|J| = r$. Hence, $r = |J| = |\psi(J)| \leq |R| = a$. Therefore, $V_{n,r} = \bigcup_{a \geq r} V_{n,r,a}$, meaning that we can bound $\Var \bigl[ \varphi \bigr]$ from above by way of upper bounds on $|V_{n,r,a}|$; the next lemma gives such a bound.

\begin{lemma} \label{Lemma:BooNRA}
Let $K$ be defined by (P5). Then
\begin{equation*}
|V_{n,r,a}| \leq |G_n| |\mathcal{A}|^{|E_n|} \biggl( \frac{ |E_n|^{4K/n} }{|\mathcal{A}|} \biggr)^a.
\end{equation*}
\end{lemma}
\begin{proof}
Let $(u,v) \in V_{n,r,a}$, and let $R_0 = R_0(u,v)$ and $R = R(u,v)$ be as above. 
By Lemma 3.8 from \cite{McGoffPavlov}, there exists a subset $\mathcal{S}$ of the cross-repeats such that $|\mathcal{S}| \leq 2 |R_0| /n$ and
\begin{equation*}
R_0 = \bigcup_{(p,q) \in \mathcal{S}} (q+F_n).
\end{equation*}
By (P5), for each $t \in E_n$, the set $\eta_n^{-1}(t)$ contains at most $K$ elements. Thus, $|R_0| \leq K |R|$, and we have that $|\mathcal{S}| \leq 2K |R| / n = 2K a/n$.

Let $\psi(u,v) = (u, \mathcal{S}, v|_{E_n \setminus R})$. By construction, $\psi$ is injective: $u$ and $\mathcal{S}$ determine $v|_{R_0}$, which then determines $v|_{R}$, and in combination with $v|_{E_n \setminus R}$, this determines $v|_{E_n}$, which determines $v$. By the obvious bounds on the image of $\psi$, we have
\begin{align*}
|V_{n,r,a}| = | \psi(V_{n,r,a}) | & \leq |G_n| \cdot \bigl[|E_n|^2 \bigr]^{2K a/n} \cdot |\mathcal{A}|^{|E_n|-a} \\
& =  |G_n| \cdot |\mathcal{A}|^{|E_n|} \biggl( \frac{ |E_n|^{4K/n} }{|\mathcal{A}|} \biggr)^a.
\end{align*}
\end{proof}


\begin{lemma} \label{Lemma:NR}
Let $\lambda>1$ be such that $|\mathcal{A}|^{-1} < \lambda^{-1} < \alpha$.
There exists a constant $C_1$ such that for all large enough $n$,
\begin{equation*}
\frac{|V_{n,r}|}{|G_n|^2} \leq C_1 \lambda^{-r}.
\end{equation*}
\end{lemma}
\begin{proof}

By the union bound, Lemma \ref{Lemma:BooNRA}, property (P1), and our choice of $\lambda$, there exists a constant $C'>0$ such that for all large enough $n$,
\begin{align*}
|V_{n,r}| \leq \sum_{a \geq r} |V_{n,r,a}| \leq |G_n| \cdot |\mathcal{A}|^{|E_n|} \sum_{a \geq r} \lambda^{-a} \leq C' |G_n| \cdot |\mathcal{A}|^{|E_n|} \lambda^{-r}.
\end{align*}
Dividing by $|G_n|^2$ and using that $|G_n|/|\mathcal{A}|^{|E_n|} \to 1$ (from Lemma \ref{Lemma:Halloween}), we see that there exists $C_1>0$ such that for all large enough $n$, 
\begin{equation*}
\frac{|V_{n,r}|}{|G_n|^2} \leq C_1 \lambda^{-r}.
\end{equation*}
\end{proof}

\begin{lemma} \label{Lemma:Candy}
There exists $C_0>1$ such that for all large enough $n$,
\begin{equation*}
\frac{|V_n|}{|G_n|^2} \leq \frac{C_0 |E_n|^2}{ |\mathcal{A}|^{n^d}}.
\end{equation*}
\end{lemma}
\begin{proof}
Let $(u,v) \in V_n$. Let $(p,q)$ be the lexicographically minimal cross-repeat. 
Let $T = \eta_n(q+F_n)$. 
By property (P6), we have $|T| = n^d$. Also, the map $(u,v) \mapsto (u,p,q,v|_{E_n \setminus T})$ is clearly injective. Hence
\begin{equation*}
|V_n| \leq |G_n| \cdot |E_n|^2 \cdot |\mathcal{A}|^{|E_n| - n^d}.
\end{equation*}
Dividing by $|G_n|^2$ and using that $|G_n|/|\mathcal{A}|^{|E_n|} \to 1$ (from Lemma \ref{Lemma:Halloween}), we see that there exists a constant $C_0 >1$ such that for all large enough $n$,
\begin{equation*}
\frac{|V_n|}{|G_n|^2} \leq \frac{C_0 |E_n|^2}{ |\mathcal{A}|^{n^d}}.
\end{equation*}
\end{proof}

We are finally prepared to prove Theorem \ref{Thm:SecondMomentMethod}.

\vspace{2mm}

\begin{ProofOfSecondMomentMethodThm}

Let $d \geq 1$, $\mathcal{A}$, $\alpha \in (|\mathcal{A}|^{-1},1]$, and $c \in (0,1)$ be as in the statement of the theorem. Additionally, let $\{\Lambda_n\}_n$ be an $\{m_n\}$-regular sequence of lattices with notation as above.
By Lemma \ref{Lemma:Halloween}, we have
\begin{equation*}
|G_n| \geq |\mathcal{A}|^{|E_n|} \Bigl( 1 - |E_n|^2 |\mathcal{A}|^{-n^d} \Bigr).
\end{equation*}
Let $a \in (c,1)$. 
Since $|E_n| = |P_n \cap \mathbb{Z}^d| = \vol(\Lambda_n)$, property (P1) (subexponential growth) implies that for all large enough $n$,
\begin{equation*}
a \Bigl( 1 - |E_n|^2 |\mathcal{A}|^{-n^d} \Bigr) \geq c.
\end{equation*}
Combining these inequalities with Equation (\ref{Eqn:Expectation}), we see that for all large enough $n$, 
\begin{align*}
a \cdot \EE \bigl[ \varphi_n \bigr] & = a \cdot \alpha^{|E_n|} \cdot |G_n| \\
 & \geq a \cdot \alpha^{|E_n|} \cdot |\mathcal{A}|^{|E_n|} \Bigl( 1 - |E_n|^2 |\mathcal{A}|^{-n^d} \Bigr)  \geq c (\alpha |\mathcal{A}| )^{|E_n|}.
\end{align*}
Using the above inequality, followed by Chebyshev's inequality and Equations (\ref{Eqn:Expectation}) and (\ref{Eqn:Variance}), we see that for all large enough $n$,
\begin{align*}
\PP \bigl( \varphi_n \leq c (\alpha |\mathcal{A}| )^{|E_n|} \bigr) & \leq 
\PP \bigl( \varphi_n \leq a \EE[\varphi_n] \bigr) \\
&  = \PP \bigl( \EE[\varphi_n] - \varphi_n \geq \EE[\varphi_n] - a \EE[\varphi_n] \bigr) \\
 & \leq \frac{\Var[ \varphi_n ]}{ \EE[\varphi_n]^2(1-a)^2 } \\
 & = \frac{\Var[ \varphi_n ]}{ |G_n|^2 \cdot \alpha^{2|E_n|} \cdot  (1-a)^2} \\
 & = \frac{1}{(1-a)^2} \sum_{r = 1}^{|E_n|} \frac{|V_{n,r}|}{|G_n|^2} \, (\alpha^{-r} -1) \\
 & \leq \frac{1}{(1-a)^2} \sum_{r = 1}^{|E_n|} \frac{|V_{n,r}|}{|G_n|^2} \, \alpha^{-r}.
\end{align*}
Let $\lambda>1$ be such that $|\mathcal{A}|^{-1} < \lambda^{-1} < \alpha$.
By Lemmas \ref{Lemma:NR} and \ref{Lemma:Candy}, we have that for all large enough $n$, 
\begin{align*}
\sum_{r = 1}^{|E_n|} \frac{|V_{n,r}|}{|G_n|^2} \alpha^{-r} & = \sum_{r = 1}^{n^d} \frac{|V_{n,r}|}{|G_n|^2} \alpha^{-r} + \sum_{r = n^d+1}^{|E_n|} \frac{|V_{n,r}|}{|G_n|^2} \alpha^{-r}  \\
 & \leq \frac{|V_n|}{|G_n|^2} \alpha^{-n^d} + C_1 \sum_{r = n^d +1}^{|E_n|} (\alpha \lambda)^{-r} \\ 
 & \leq \frac{C_0 |E_n|^2}{ |\mathcal{A}|^{n^d}} \alpha^{-n^d} + C_2 (\alpha \lambda)^{-n^d}.
\end{align*}
Combining the two previous displays and again using property (P1) (subexponential growth), we obtain that there exists $\rho > 0$ (indeed we could choose any $\rho < \log(\alpha \lambda)$) such that for all large enough $n$,
\begin{equation*}
\PP \bigl( \varphi_n \leq c (\alpha |\mathcal{A}| )^{|E_n|} \bigr)  \leq e^{- \rho n^d},
\end{equation*}
as desired.
\end{ProofOfSecondMomentMethodThm}


\section{Factors} \label{Sect:Factors}

In this section, we use the many patterns with prescribed repeat structure guaranteed by Theorem \ref{Thm:SecondMomentMethod} to construct the desired surjective factor maps. 
The main construction of the factor map is done in the same way as in \cite{McGoffPavlov_factors}. 
Instead of repeating all arguments, we only describe how they should be adapted to prove Theorem \ref{Thm:Factors}. 
We note that the arguments given in \cite{McGoffPavlov_factors} were technically given only for $d > 1$, since that was the only novel case in that work. Nonetheless, they also apply to the case $d = 1$, and so we use them in that case as well.

\vspace{2mm}

\begin{ProofOfFactorThm}

Assume that $Y$, $\alpha = \alpha_X$, and $\mathcal{A} = \mathcal{A}_X$ satisfy the hypotheses of the theorem, and choose $g$ so that $Y$ has the $g$-extension property for a set of finite patterns with diameters bounded from above by $g$. The construction of our factor map is very similar to the one constructed in 
\cite{McGoffPavlov_factors}. However, in that work, we assumed block gluing on the domain $X$, which we have no reason to suspect holds for random $\mathbb{Z}^d$ SFTs with high probability. The proof from that work is too complicated to describe in full here; instead we will only informally define the factor map $\phi$, describing how we obviate the need for the block gluing assumption. The proof from \cite{McGoffPavlov_factors} begins with the construction of a family of ``marker'' words in $\mathcal{L}(X)$, with restrictions on possible overlaps between marker words. Knowledge of marker words appearing in a point of $X$, along with their locations, completely determines $\phi(x)$. In \cite{McGoffPavlov_factors}, these markers are constructed using block gluing. However, in the end, the reader can check that all that was required was the existence of a set $S$ with the following properties for some parameters $k,m$. (The first parameter $k$ was the side length of a so-called ``surrounding frame'' in \cite{McGoffPavlov_factors}, and in the notation of that work, was written as 
$k + 2m + 2g$. The use of $m$ to denote the second parameter is not accidental; it will be the same $m$ for which we apply the second moment method.)

\begin{itemize}
\item[(i)] $S$ is a set of patterns on $F_k$ which can be partitioned as $$\displaystyle S = \bigsqcup_{i=1}^M S_i,$$ where 
$M > |\mathcal{L}_{F_{k-m-g}}(Y)| \cdot |\mathcal{A}_Y|^{d((k+m+(2d+3)g)^d-(k+m-(2d-1)g)^d)}$. 


\item[(ii)] For any choice of $y \in [1,M]^{\mathbb{Z}^d}$, there exists $x \in X$ such that $x((k-m)v+F_k) \in S_{y(v)}$ for all $v \in \mathbb{Z}^d$.

\item[(iii)] Given $v \neq v' \in \mathbb{Z}^d$ so that $x(v + F_k), x(v' + F_k) \in S$, either $v' - v = (k - m) u$ for some $u \in \Z^d$ with $\|u\|_{\infty} \leq 1$, or $\|v' - v\|_{\infty} > k - m + (2d+3)g$.
\end{itemize}

For ease of notation, we define $p(k) = d((k+m+(2d+3)g)^d-(k+m-(2d-1)g)^d)$, and note that $p$ has degree $d-1$ in $k$.

Given a set $S$ and parameters $k$ and $m$ satisfying these properties, the factor map $\phi$ behaves as follows. 
For any $x \in X$, $\phi$ finds all $v \in \mathbb{Z}^d$ for which $x(v + F_k) \in S$, and then, for each such $v$, the 
set $S_i$ in which $x(v + F_k)$ lies determines choices of letters on $v + F_{k-m-g}$. (We follow \cite{McGoffPavlov_factors} and call $v + F_{k-m-g}$ the \textbf{determined zone} associated to $x(v + F_k)$). Then letters are placed and erased in a series of steps, all completely determined by the location of words in $S$ and the $S_i$ each lies in, until finally all of $\mathbb{Z}^d$ is filled, yielding $\phi(x)$. 
This portion of the proof uses only properties of $Y$ and will be completely unchanged from the proof in \cite{McGoffPavlov_factors} once $S$ has been determined.

Then, (i) and (ii) together are used to prove surjectivity. Indeed, (i) yields a surjection from $[1,M]$ to $\mathcal{L}_{F_{k-m-g}}(Y) \cdot \mathcal{A}_Y^{p(k)}$. Then, given any $y \in Y$, a point $x \in X$ as in (ii) can be defined with patterns from properly chosen 
$S_i$ that cause the proper letters to be placed and erased in subsequent steps, yielding $\phi(x) = y$. 

Therefore, omitting the elements of the proof that may be repeated exactly as in \cite{McGoffPavlov_factors}, we need only prove that for properly chosen parameters $k$ and $m$ (depending on $n$), there exists $\rho>0$ such that for all large enough $n$,
\begin{equation*}
\mathbb{P}_{n,\alpha}\bigl(\exists S \textrm{ that avoids $\mathcal{F}$  and satisfies (i), (ii), and (iii)} \bigr)  \geq 1 - e^{-\rho n^d}.
\end{equation*}
Let us now prove this statement.

Choose any $k = k(n)$ satisfying $\frac{k}{n} \rightarrow \infty$ and $\frac{\log k}{n} \rightarrow 0$, and take $m = m(n) = n + (2d+3)g$. Following the notation of Section \ref{Sect:SecondMomentMethod}, for every $n$ define the lattice $\Lambda_n = (k-m)\mathbb{Z}^d$, with associated basis $\{(k-m)e_i\}_{i=1}^d$; the sequence $\Lambda_n$ is then $\{m_n\}$-regular. In this case, the associated objects from Section \ref{Sect:SecondMomentMethod} are as follows. We have $D_n = [-m+1, k]^d \subset \mathbb{Z}^d$, and $E_n = [1, k-m]^d \subset \mathbb{Z}^d$. Also, the set $G_n$ consists of all patterns $u \in \mathcal{A}^{D_n} = \mathcal{A}^{[-m+1,k]^d}$ where for any $A,B \in \mathcal{C}_n(D_n)$, $u|_A = u|_{B}$ iff $A = p + B$ for some $p \in \Lambda_n$; in particular, $|W_n(u)| = |E_n| = (k-m)^d$.
Recall that for a set of forbidden patterns $\mathcal{F} \subset \mathcal{A}^{F_n}$, we let $G_n(\mathcal{F})$ denote the patterns in $G_n$ that avoid the forbidden patterns in $\mathcal{F}$.

For the moment, fix a set of forbidden words $\mathcal{F} \subset \mathcal{A}^{F_n}$. 
We truncate all elements of $G_n(\mathcal{F})$ to $F_k$, resulting in a collection $G_n'(\mathcal{F})$, and clearly $|G_n'(\mathcal{F})| \geq |\mathcal{A}|^{k^d - (k+m)^d} |G_n(\mathcal{F})|$. Note that for $w \in G_n'(\mathcal{F})$, it is still the case that for $A$ and $B$ in $\mathcal{C}_n(D_n)$, we have $w|_A = w|_B$ if and only if $A = p+B$ for some $p \in \Lambda_n$. Moreover, it is possible to choose a subcollection $G_n''(\mathcal{F}) \subset G_n'(\mathcal{F})$ such that all patterns in $G_n''(\mathcal{F})$ share the same subpattern on 
$\partial^{in}_m(F_k)$ and 
\begin{align} 
\begin{split} \label{Eqn:Ducks}
|G_n''(\mathcal{F})| & \geq |\mathcal{A}|^{-|\partial^{in}_m(F_k)|} |G_n'(\mathcal{F})| \\
& = |\mathcal{A}|^{(k-2m)^d - k^d} |G_n'(\mathcal{F})| \\
& \geq |\mathcal{A}|^{(k-2m)^d - (k+m)^d} |G_n(\mathcal{F})| \\
& \geq |\mathcal{A}|^{-3dkm} |G_n(\mathcal{F})|.
 \end{split}
\end{align}

Now we consider the set of forbidden words $\mathcal{F}$ to be chosen at random with distribution given by $\mathbb{P}_{n,\alpha}$.
By Theorem \ref{Thm:SecondMomentMethod}, since $\log(\alpha |\mathcal{A}|) > h(Y)$, there exists $\epsilon > 0$ and $\rho>0$ such that for all large enough $n$,  
\begin{equation}\label{eqn1}
\mathbb{P}_{n,\alpha} \bigl(|G_n(\mathcal{F})| > e^{k^d (h(Y) + \epsilon)}\bigr) > 1 - e^{-\rho n^d}.
\end{equation}
Then by (\ref{Eqn:Ducks}) and (\ref{eqn1}) and inclusion, we have that for all large enough $n$,
\begin{equation*}
\mathbb{P}_{n,\alpha} \bigl(|G_n''(\mathcal{F}) | > e^{k^d (h(Y) + \epsilon) - 3dkm \log |\mathcal{A}|} \bigr) > 1 - e^{- \rho n^d}.
\end{equation*}

Since $p$ has degree $d-1$, it is the case that $e^{k^d (h(Y) + \epsilon) - 3dkm \log |\mathcal{A}|} > |\mathcal{L}_{[1,k-m-g]^d}(Y)| \cdot |\mathcal{A}_Y|^{p(k)}$ for large $n$, and so for such $n$,
\begin{equation}\label{eqn2}
\mathbb{P}_{n,\alpha} \bigl(|G_n''(\mathcal{F}) | > |\mathcal{L}_{[1,k-m-g]^d}(Y)| \cdot |\mathcal{A}_Y|^{p(k)} \bigr) > 1 - e^{- \rho n^d}.
\end{equation}

It remains only to show that on the event $|G_n''(\mathcal{F})| > |\mathcal{L}_{F_{k-m-g}}(Y)| \cdot |\mathcal{A}_Y|^{p(k)}$, the set $S = G_n''(\mathcal{F})$ is a collection satisfying (i), (ii), and (iii). By taking each $S_i$ to be a singleton, (i) is obviously satisfied. Since all patterns in $G_n''(\mathcal{F})$ agree on $\partial^{in}_m(F_k)$, $m \geq n$, and $X$ is an SFT defined by forbidden words with shape $F_n$, (ii) is satisfied.

To prove (iii), we may clearly assume without loss of generality (by translating if necessary) that $v' = 0$. We then assume for a contradiction that $x(F_k)$ and $x(v + F_k)$ are in $G_n''(\mathcal{F})$ (after translation), $v \neq 0$, $v$ is not of the form $(k - m)u$ for some $u$ with $\|u\|_{\infty} \leq 1$ (i.e. $v \notin \Lambda_n$), and $\|v\|_{\infty} \leq k - m + (2d+3)g = k - n$. We assume for now that all coordinates of $v$ are nonnegative; since our arguments are not affected by reflecting over any plane $x_i = 0$, other cases are similar. We then know that $0 \leq v_i \leq k - n$ for $1 \leq i \leq d$.

Recall that by definition of $G_n''(\mathcal{F})$, we have $x(\partial^{in}_m(F_k)) = x(v + \partial^{in}_m(F_k))$. 
Since $m \geq n$, $F_n \subseteq \partial^{in}_m(F_k)$, and so $x(F_n) = x(v + F_n)$. On the other hand, since $0 \leq v_i \leq k - n$ for $1 \leq i \leq d$, $v + F_n \subseteq F_k$. This means that $x(F_n)$ and $x(v + F_n)$ are equal subwords of $x(F_k) \in G_n''(\mathcal{F})$ whose shapes differ by a vector not in $\Lambda_n$, a contradiction to the definition of $G_n''(\mathcal{F})$. Therefore, (iii) is proved, and so the collection $S = G_n''(\mathcal{F})$ satisfies (i), (ii), and (iii).

\end{ProofOfFactorThm}

\section{Embeddings} \label{Sect:Embeddings}

Here we use the many patterns with prescribed repeat structure that are guaranteed by Theorem \ref{Thm:SecondMomentMethod} to construct the desired embeddings.

\vspace{2mm}

\begin{ProofOfEmbeddingThm}

Assume that $X$, $\alpha = \alpha_Y$, and $\mathcal{A} = \mathcal{A}_Y$ are as in the theorem and that $\{\mathcal{O}_n\}_n$ and $\{\Lambda_n\}_n$ are the orbits and associated lattices guaranteed by the fact that $X$ satisfies the periodic marker condition for $m_n = n$. Our embedding map will be a much simpler version of the ones from \cite{Lightwood2003,Lightwood2004}. In our setting, the assumption that $X$ factors onto finite orbits replaces the much more difficult marker construction in that previous work.

Again we follow the notation from Section~\ref{Sect:SecondMomentMethod}, i.e. $P_n$ is the fundamental domain of $\Lambda_n$, $D_n$ is the set of $q \in \mathbb{Z}^d$ with $d(q,P_n) \leq n$, $\eta_n: D_n \rightarrow P_n$ sends any $p$ to the unique element $q \in P_n$ for which $p-q \in \Lambda_n$, and $E_n = \eta_n(D_n) = P_n \cap \mathbb{Z}^d$. By the small outer boundaries property (P2), 
we have
$\frac{|E_n|}{|D_n|} \rightarrow 1$. 

As usual, let $G^0_n$ be the set of patterns $w \in \mathcal{A}^{D_n}$ where $w(A) = w(B)$ whenever $A, B \in C_n(D_n)$ and $A = p + B$ for some $p \in S$, and define $G_n \subseteq G^0_n$ to be the set of all patterns $u \in G^0_n$ with $|W_n(u)| = |E_n|$, i.e. where the only equal pairs of subpatterns in $u$ with shape $F_n$ are those guaranteed by the definition of $G^0_n$. Then, by Theorem \ref{Thm:SecondMomentMethod}, since $\log(\alpha |\mathcal{A}|) > h(X)$, there exist $\epsilon > 0$ and $\rho>0$ so that for all large enough $n$, 
\begin{equation}\label{eqn2.5}
\mathbb{P}_{n,\alpha} \bigl(|G_n(\mathcal{F})| > e^{(h(X) + \epsilon)|E_n|} \bigr) > 1 - e^{- \rho n^d}.
\end{equation}
Define $D'_n$ to be the set of $q \in D_n$ with $d(q, \mathbb{Z}^d \setminus D_n) \leq 2n$. We claim that $D'_n \subseteq (D_n \setminus E_n) + \{-n,0,n\}^d$. To see this, take any $q \in D'_n$. There are two cases. If $d(q, \mathbb{Z}^d \setminus D_n) \leq n$, then $q \in D_n \setminus E_n$ by definition. If instead $d(q, \mathbb{Z}^d \setminus D_n) > n$, then $q + [-n,n]^d \subseteq D_n$, but there exists $r \in \mathbb{Z}^d \setminus D_n$ with $d(q,r) \leq 2n$. Then, define $s \in 
\{-n,n\}^d$ to have each coordinate with the same sign as the corresponding coordinate of $r - q$. Then $q + s \in D_n$ since $s \in [-n,n]^d$, and it is easily checked that $d(q+s, r) \leq n$. Then $q + s \in D_n \setminus E_n$, and so $q \in (D_n \setminus E_n) + \{-n,0,n\}^d$, completing the proof.

This implies that $|D'_n| \leq 3^d |D_n \setminus E_n|$, and so since $\frac{|E_n|}{|D_n|} \rightarrow 1$, we have $\frac{|D'_n|}{|E_n|} \rightarrow 0$. 

For any $\mathcal{F} \subset \mathcal{A}^{F_n}$, it is possible to choose a subcollection $G_n'(\mathcal{F}) \subset G_n(\mathcal{F})$ where all patterns in $G_n'(\mathcal{F})$ share the same subpattern on $D'_n$ and $|G_n'(\mathcal{F})| > |\mathcal{A}|^{-|D'_n|} |G_n(\mathcal{F})|$. 

Since the sets $E_n$ are intersections of convex subsets with $\mathbb{Z}^d$ and have inradii approaching infinity by property (P3) (small inner boundaries), Theorem A from \cite{BBQ} implies that $\lim \frac{\log |\mathcal{L}_{E_n}(X)|}{|E_n|} = \log h(X)$. Therefore, for large enough $n$, we have $|\mathcal{A}|^{-|D'_n|} e^{(h(X) + \epsilon) |E_n|} > |\mathcal{L}_{E_n}(X)|$, and so (\ref{eqn2.5}) implies that for such $n$, 
\begin{equation}\label{eqn3}
\mathbb{P}_{n,\alpha} \bigl(|G_n'(\mathcal{F})| > |\mathcal{L}_{E_n}(X)| \bigr) > 1 - e^{-\rho n^d}.
\end{equation}

It now suffices to show that for any $\mathbb{Z}^d$ SFT $Y$ defined by forbidden patterns on $F_n$ and for which one can choose $G' = G_n'(\mathcal{F})$ as above with $|G'| > |\mathcal{L}_{E_n}(X)|$, there exists an embedding $\psi$ from $X$ to $Y$. Assume that $Y$ is such an SFT, and define an injection $\gamma$ from $\mathcal{L}_{E_n}(X)$ to $G'$. Then, denote by $\beta$ the assumed factor map from $X$ onto the periodic orbit $\mathcal{O}_n$, and fix an element $z \in \mathcal{O}_n$. We may think of $\beta$ as assigning, in a shift-commuting and continuous way, an element of $E_n$ to each $x \in X$, namely the unique $v \in E_n$ for which $\beta(x) = \sigma_v z$. Given $x$, we then define $\psi(x)$ as follows: for every $v \in \Lambda_n$, $(\psi(x))(v + \beta(x) + D_n)$ is assigned to be $\gamma(x(v + \beta(x) + E_n))$. 

We first must check that $\psi$ is well-defined, since the sets $\{v + \beta(x) + D_n\}_{v \in \Lambda_n}$ are not disjoint. 
Suppose that some $t \in \mathbb{Z}^d$ is contained in both $v + \beta(x) + D_n$ and $v' + \beta(x) + D_n$ for $v \neq v' \in \Lambda_n$, and let $w,w' \in G_n'$ be the patterns assigned to $\psi(x)$ on the shapes $v + \beta(x) + D_n$ and $v' + \beta(x) + D_n$, respectively. We must show that $w$ and $w'$ assign the same letter at $t$, i.e. that $w(t - v - \beta(x)) = w'(t - v' - \beta(x))$. Note that since $w,w' \in G'$, we have $w(D'_n) = w'(D'_n)$. 

We first observe that since $t \in v' + \beta(x) + D_n$, the definition of $D_n$ gives that $d(t, v' + \beta(x) + E_n) \leq n$. Since $v' + \beta(x) + E_n$ and $v + \beta(x) + E_n$ are disjoint, this implies that $d(t, \mathbb{Z}^d \setminus (v + \beta(x) + D_n)) \leq 2n$. Thus $t \in v + \beta(x) + D'_n$, and then $t - v - \beta(x) \in D'_n$. Therefore, $w(t - v - \beta(x)) = w'(t - v - \beta(x))$. Since $t - v' - \beta(x)$ is also in $D'_n$, and since $v - v' \in \Lambda_n$, by definition of $G_n$ it must be the case that $w'(t - v - \beta(x)) = w'(t - v' - \beta(x))$, and so 
$w(t - v - \beta(x)) = w'(t - v' - \beta(x))$, as desired.

So, $\psi$ is well-defined, and it's shift-commuting and continuous since $\beta$ is. It remains only to prove that $\psi(x) \in Y$ and to check injectivity. To prove that $\psi(x) \in Y$, we first claim that every translate of $F_n$ is contained entirely within $v + \beta(x) + D_n$ for some $v \in \Lambda_n$. To this end, consider any set $p + F_n$. Since the sets $\{v + \beta(x) + E_n\}_{v \in \Lambda_n}$ partition $\mathbb{Z}^d$, there exists $v$ so that $p \in v + \beta(x) + E_n$. But then by definition of $D_n$, we have $p + F_n \subseteq v + \beta(x) + D_n$. Now, for each $p + F_n$, we know that $(\psi(x))(p + F_n)$ is a subpattern of some $(\psi(x))(v + \beta(x) + D_n)$, and thus in $\mathcal{L}(Y)$. Since $Y$ is an SFT determined by forbidden configurations with shape $F_n$, we then know that $\psi(x) \in Y$. 

Finally, we will verify that $\psi$ is injective, and to that end, we suppose that $x_1 \neq x_2$ are points in $X$. There are two cases. For the first case, suppose that $\beta(x_1) = \beta(x_2)$ and denote their common value by $\beta$. Since the sets $\{v + E_n\}_{v \in \Lambda_n}$ partition $\mathbb{Z}^d$, there exists $v \in \mathbb{Z}^d$ so that $(x_1)(v + \beta + E_n) \neq (x_2)(v + \beta + E_n)$, and so by injectivity of $\gamma$, we see that $(\psi(x_1))(v + \beta + D_n) \neq (\psi(x_2))(v + \beta + D_n)$, meaning that $\psi(x_1) \neq \psi(x_2)$. 

For the second case, suppose that $\beta(x_1) \neq \beta(x_2)$. Then by definition of $\psi$, for all $v \in \Lambda_n$ (including $v = 0$), the patterns $(\psi(x_1))(v + \beta(x_1) + D_n)$ and $(\psi(x_2))(v + \beta(x_2) + D_n)$ are both in $G'$ (after shifting to have shape $D_n$). Define $t = \eta_n(\beta(x_2) - \beta(x_1))$; then $t \in E_n \setminus \{0\}$. Then, $(\psi(x_1))(\beta(x_1) + D_n)$ and $(\psi(x_2))(t + \beta(x_1) + D_n)$ are in $G'$, and so $(\psi(x_1))(\beta(x_1) + D'_n))$ and $(\psi(x_2))(t + \beta(x_1) + D'_n)$ are the same word. Since $t \in E_n \setminus \{0\}$, the sets $E_n$ and $E_n + t$ intersect nontrivially, and so there exists $u \in \partial^{in}_1(\beta(x_1) + E_n) \cap \partial^{in}_1(t + \beta(x_1) + E_n)$. Consider the set $u + F_n$. All sites in $u + F_n$ are within distance $n$ of $u$, which itself is within distance $n$ of both $\mathbb{Z}^d \setminus(\beta(x_1) + D_n)$ and $\mathbb{Z}^d \setminus(t + \beta(x_1) + D_n)$, and so $u + F_n \subset (\beta(x_1) + D'_n) \cap (t + \beta(x_1) + D'_n)$. This clearly implies that $t + u + F_n \subset t + \beta(x_1) + D'_n$, and so $(\psi(x_1))(u + F_n) = (\psi(x_2))(t + u + F_n)$. We also note that $u + F_n, t + u + F_n \subseteq t + \beta(x_1) + D'_n$, and it is not the case that $\eta_n(u) = \eta_n(t + u)$ since $t \in E_n \setminus \{0\}$. Therefore, by definition of $G_n$, 
$(\psi(x_2))(u + F_n) \neq (\psi(x_2))(t + u + F_n)$. However, this implies that $(\psi(x_1))(u + F_n) \neq (\psi(x_2))(u + F_n)$, and so $\psi(x_1) \neq \psi(x_2)$ as desired. 

In both cases, we have shown that $\psi(x_1) \neq \psi(x_2)$, and so $\psi$ is injective. We have verified that $\psi$ is an embedding which exists for any random $\mathbb{Z}^d$ SFT $Y$ with $|G_n'(\mathcal{F})| > |\mathcal{L}_{E_n}(X)|$, and so by (\ref{eqn3}), we are done.

\end{ProofOfEmbeddingThm}

\section{Discussion}

We here discuss a few questions and directions for further work in this area.

Theorem \ref{Thm:Embeddings} establishes that if a $\mathbb{Z}^d$ subshift $X$ satisfies the periodic marker condition and a certain entropy inequality, then $X$ may be embedded into a random $\mathbb{Z}^d$ SFT $Y$ with probability tending to one. It is natural to compare this to the best known embedding results for $\mathbb{Z}^d$ subshifts, which were obtained by Lightwood (\cite{Lightwood2003,Lightwood2004}). Those results impose a deterministic uniform mixing assumption on the codomain, but they only require that the domain 
$X$ has no periodic points. To prove these results, Lightwood shows that if $X$ has no periodic points, then it is possible to find a certain marker structure for points of $X$. More formally, each point of $X$ can be associated to a tiling of $\mathbb{R}^d$ with a finite number of polytope prototiles. The geometry of these tiles is in general much more complicated than the parallelotopes given by our periodic marker condition, and we were unable to adapt the second moment argument from Section~\ref{Sect:SecondMomentMethod} to this more general setting. 

\begin{question}
Is it possible to use Lightwood's more general marker construction to produce an embedding from an aperiodic $\mathbb{Z}^d$ SFT $X$ into a random $\mathbb{Z}^d$ SFT with limiting probability one? 
\end{question}

We also note that Theorems~\ref{Thm:Factors} and \ref{Thm:Embeddings} are similar to two questions that we did not address, namely those of embedding a random $\mathbb{Z}^d$ SFT into a fixed $\mathbb{Z}^d$ subshift and of factoring a fixed $\mathbb{Z}^d$ subshift onto a random $\mathbb{Z}^d$ SFT. Both have some associated obstacles which we do not yet know how to overcome. 

\begin{question}\label{badq1}
Are there hypotheses on $|\mathcal{A}_X|$ and $\alpha_X$ and a fixed $\mathbb{Z}^d$ subshift $Y$ that guarantee that the random $\mathbb{Z}^d$ SFT $X$ embeds into $Y$ with limiting probability $1$?
\end{question}

We can immediately obtain some necessary conditions on such a $Y$. For every $k$, define the set $P_k$ of points in $\mathcal{A}_X^{\mathbb{Z}^d}$ that have period set containing $ke_i$ for $1 \leq i \leq d$. By Lemma~\ref{perlemma}, for every $k$, the limiting probability of $X$ containing $P_k$ is at least $\alpha^{k^d |P_k|} > 0$. Therefore, for a $Y$ as in Question~\ref{badq1}, every $P_k$ must embed into $Y$, meaning that for every $k$, $Y$ has at least $|\mathcal{A}_X|^{k^d}$ points with period set containing $ke_i$ for $1 \leq i \leq d$. Since such points are in one-to-one correspondence with their restriction to $F_k$, this means that $|\mathcal{L}_{F_k}(Y)| \geq |\mathcal{A}_X|^{k^d}$ for every $k$, implying that $h(Y) \geq \log |\mathcal{A}_X|$, regardless of the value of $\alpha_X$.

\begin{question}\label{badq2}
Are there hypotheses on $|\mathcal{A}_Y|$ and $\alpha_Y$ and a fixed $\mathbb{Z}^d$ subshift $X$ that guarantee that $X$ factors onto the random $\mathbb{Z}^d$ SFT $Y$ with limiting probability $1$?
\end{question}

An $X$ as in Question~\ref{badq2} could not have any points with finite orbit at all: if $X$ contained a point with finite orbit of size $k$, then $Y$ would have to contain a point with finite orbit with size less than or equal to $k$ with limiting probability $1$, contradicting Lemma~\ref{perlemma}. Also, in all previous work on (surjective) factor maps for $\mathbb{Z}^d$ subshifts, the domain has a uniform mixing property; we were able to substitute ``typicality'' in Theorem~\ref{Thm:Factors}, but clearly cannot in the case of Question \ref{badq2}. Unfortunately, there are no known examples of $\mathbb{Z}^d$ subshifts $X$ with uniform mixing properties and no points with finite orbit.

Finally, we note that in this work, we did not fully address the topic of embeddings/factor maps between two random $\mathbb{Z}^d$ SFTs $X,Y$ due to the unavoidable positive limiting probability that the necessary periodic point condition between $X,Y$ would fail. We do not know the answer to the following question.

\begin{question}
Does an embedding/factor map exist from $X$ to $Y$ with limiting conditional probability $1$ on the the event that $X,Y$ satisfy the appropriate necessary periodic point condition?
\end{question}

\bibliographystyle{plain}
\bibliography{RandomSFTs_maps_refs}

\end{document}